\documentclass[12pt,letterpaper]{amsart}

\usepackage{amsmath}
\usepackage{amsfonts}
\usepackage{amsthm}
\usepackage{amssymb}
\usepackage{mathrsfs}
\usepackage{cite}
\usepackage{tikz}

\newif\ifshowvc
\showvctrue
\showvcfalse  
\makeatletter
\def\blfootnote{\xdef\@thefnmark{}\@footnotetext}
\makeatother
\ifshowvc
\input{vc}
\fi

\theoremstyle{plain}
\newtheorem{theorem}{Theorem}[section]
\newtheorem{proposition}[theorem]{Proposition}

\newtheorem{lemma}[theorem]{Lemma}

\newtheorem{conj}[theorem]{Conjecture}

\theoremstyle{definition}
\newtheorem{definition}[theorem]{Definition}

\theoremstyle{remark}
\newtheorem{remark}[theorem]{Remark}
\newtheorem{example}[theorem]{Example}

\newcommand{\beq}{\begin{equation*}}
\newcommand{\eeq}{\end{equation*}}
\newcommand{\mb}[1]{\mathbf{#1}}
\newcommand{\BBB}{\mathscr{B}}
\newcommand{\OOO}{\mathscr{O}}

\newcommand{\x}{\mathbf{x}}
\newcommand{\z}{\mathbf{z}}
\newcommand{\p}{\varpi }
\newcommand{\C}{\mathbb{C}}
\newcommand{\R}{\mathbb{R}}
\newcommand{\Z}{\mathbb{Z}}
\newcommand{\fg}{\mathfrak{g}}

\newcommand{\ul}[1]{\underline{#1}}
\renewcommand{\hat}{\widehat}

\DeclareMathOperator{\GL}{GL}

\DeclareMathOperator{\sign}{sgn}
\DeclareMathOperator{\bx}{\mathtt{bx}}

\begin{document}

\title[Crystal graphs and $G_2$]{Crystal graphs, Tokuyama's theorem,
and the Gindikin--Karpelevi\v c formula for $G_{2}$}
\author[H. Friedlander]{Holley Friedlander}
\address{Department of Mathematics and Statistics\\
Williams College\\
Williamstown, MA 01267}
\email{hf2@williams.edu}

\author[L. Gaudet]{Louis Gaudet}
\address{Department of Mathematics\\
Yale University\\
New Haven, CT}
\email{louis.gaudet@yale.edu}

\author[P. E. Gunnells]{Paul
E. Gunnells}
\address{Department of Mathematics and Statistics\\
University of Massachusetts\\
Amherst, MA  01003}
\email{gunnells@math.umass.edu}

\thanks{HF and PG thank the NSF for support.  LG thanks the Yale
Mathematics Department for support.}


\begin{abstract}
We conjecture a deformation of the Weyl character formula for type $G_{2}$
in the spirit of Tokuyama's formula for type $A$.  Using our
conjecture we prove a combinatorial version of the
Gindikin--Karpelevi\v c formula for $G_{2}$, in the spirit of
Bump--Nakasuji's formula for type $A$.
\end{abstract}

\maketitle
\ifshowvc
\blfootnote{Base revision~\GITAbrHash, \GITAuthorDate, \GITAuthorName.}
\fi

\section{Introduction}\label{s:intro}

Let $\fg$ be a simple complex Lie algebra, let $\Lambda_{W}$ be its
weight lattice, and let $\C [\Lambda_{W}]$ be the associated ring of
Laurent polynomials.  Let $W$ be the Weyl group of $\fg$ and for any
$w\in W$ let $\sign w \in \{\pm 1 \}$ be its sign.  Given a dominant
weight $\theta \in \Lambda_{W}$, let $V_{\theta}$ be the irreducible
representation of highest weight $\theta$.  The Weyl character formula
expresses the character $\chi_{\theta} \in \C [\Lambda_{W}]$ as a ratio
of two polynomials:
\begin{equation}\label{eq:wcf0}
\chi_{\theta} (\x) = \frac{\sum_{w\in W} (\sign w) \x^{w (\theta +\rho)-\rho
}}{\prod_{\alpha > 0} (1-\x^{-\alpha})}.
\end{equation}
Here the product is taken over the positive roots $\alpha$, the Weyl
vector $\rho$ is $\frac{1}{2}\sum_{\alpha >0} \alpha $, and for any weight
$\beta$ we denote by $\x^{\beta}$ the corresponding monomial in $\C
[\Lambda_{W}]$.

We can define a deformation of \eqref{eq:wcf0} by inserting a
parameter into the denominator.  Let $q$ be a variable and put
\[
D (\x) = \prod_{\alpha >0} (1-q^{-1}\x^{-\alpha}).  
\]
Then the product 
\[
N_{\theta} (\x) = \chi_{\theta} (\x ) D (\x) 
\]
is a polynomial supported in the convex hull of the
weights of the representation $V_{\theta +\rho}$.  When $\fg$ has type
$A$, Tokuyama \cite{tokuyama} showed how to compute $N_{\theta} (\x)$
explicitly as a sum over the Gel$'$fand--Cetlin basis of $V_{\theta
+\rho}$.  His formula has recently played an imporant role in the
study of Weyl group multiple Dirichlet series.  These are series in
several complex variables built from data attached to root systems;
each has a group of functional equations isomorphic to the Weyl group
of the root system that intermixes all the variables.  Such series are
related to $p$-adic Whittaker functions and in fact are conjectured
to be Fourier--Whittaker coefficients of certain Eisenstein series on
metaplectic groups (finite central covers of reductive groups).  We
refer to \cite{bump.intro} for more information about this connection.

Tokuyama's formula has been generalized to other root systems with
various combinatorial tools.  For instance Hamel--King
\cite{hamel.king} gave a generalization to $\fg$ of type $C$, in which
the Gel$'$fand--Cetlin basis was replaced by symplectic shifted
tableaux.  Conjectural generalizations to $\fg$ of types $B$, $D$ were
given in \cite{spincrystal, d4conjecture, brconjecture}; recently the
case of type $B$ was proved by Friedberg--Zhang \cite{fz}.  For
arbitrary $\Phi$, the most general result is due to McNamara
\cite{mcnamara}, who showed how $p$-adic Whittaker functions can be
computed as sums over crystal graphs.(\footnote{Another approach also
valid for an arbitrary Cartan--Killing type has been presented by
Kim--Lee \cite{kim.lee}, who compute $N_{\theta } (\x)$ as a sum over
weights of $V_\theta \otimes V_{\rho}$.})  When $\fg$ is type $A$, the
sums can be taken over Gel$'$fand--Cetlin patterns and computed
explicitly, and McNamara recovers Tokuyama's theorem.  However, apart
from this case, McNamara's formulas have not been explicitly computed
for any other type.

In this paper, we present a conjectural analogue of Tokuyama's theorem
when $\fg$ has type $G_{2}$ (Conjecture \ref{conj:g2conj}).  We
describe how to compute the polynomial $N_{\theta} (\x)$ as a sum over
certain weight vectors in $V_{\theta+\rho }$.  As a combinatorial
model for this representation, we use patterns due to Littelmann
\cite{littelmann}; when $\fg$ has type $A$, these are equivalent to
Gel$'$fand--Cetlin patterns.  Although we are unable to prove our
conjecture, we are able to treat the limiting case that the highest
weight becomes infinite.  In this case our formula (Theorem
\ref{thm:G2GK}) becomes a combinatorial version of the
Gindikin--Karpelevi\v c formula \cite{langlands.eulerproducts}, in the
spirit of that proved by Bump--Nakasuji \cite{bn}.
 
\section{Background and the Tokuyama numerator}

In this section we state Tokuyama's formula for characters of
representations of $\GL_{r+1}$ and explain the connection to crystal
graphs.  We begin by describing what a formula of
``Tokuyama-type'' looks like.  We will use slightly different
normalizations from \S \ref{s:intro}: in particular we will shift our
characters so that they are supported on the root lattice, and will
index representations by lowest weights.  These conventions are
somewhat unusual from the point of view of combinatorial
representation theory, but they are more natural when one connects
these constructions to $p$-adic Whittaker functions.

As before let $\fg$ be a simple complex Lie algebra of rank $r$.  Let
$\Phi$ be the root system of $\fg$ and $\Phi^{+} \cup \Phi^{-}$ the
partition into positive and negative roots, and $\Delta =
\{\alpha_{1},\dotsc ,\alpha_{r}\}$ the simple roots.  Let
$\varpi_{1},\dotsc ,\varpi_{r}$ be the fundamental weights and $\rho =
\frac{1}{2}\sum_{\alpha >0}\alpha = \sum \varpi_{i}$ . Let $W$ be the
Weyl group of $\Phi$ with simple reflections $s_{1},\dotsc ,s_{r}$.

We let $\Lambda$ be the lattice generated by the roots and $\C
[\Lambda] $ the ring of Laurent polynomials determined by $\Lambda$.
Given $\lambda \in \Lambda$, let $\x^{\lambda}\in \C [\Lambda]$ be the
corresponding monomial.  We may identify $\C [\Lambda]$ with $\C
[x_{1}^{\pm 1}, \dotsc , x_{r}^{\pm 1}]$ via $\x^{\alpha_{i}} \mapsto
x_{i}$.  Let $\Lambda^{+}\subset \Lambda$ be the cone generated by the
positive roots (the codominant cone).

Let $q$ be a parameter.  We define the Weyl denominator by 
\[
\Delta (\x) = \prod_{\alpha > 0} (1- \x^{\alpha})
\]
(note the use of $\x^{\alpha}$, not $\x^{-\alpha}$) and a deformation $D(\x)$
of $\Delta (\x )$ by
\begin{equation}\label{eqn:def}
D (\x) = \prod_{\alpha >0} (1- q^{-1}\x^{\alpha}).
\end{equation}

Let $\theta$ be a dominant weight and let $V_{\theta}$ be the
irreducible representation of $\fg$ with \emph{lowest weight
$-\theta$}.(\footnote{For many root systems, including $G_{2}$, the
representation $V_{\theta}$ as defined coincides with the
representation with highest weight $\theta$.  For some, such as type
$A$, they differ.  This choice means that certain changes have to be
made when comparing results we cite below with the original sources.})
Let $\chi_{\theta}$ be the character of $V_{\theta}$.  As in \S
\ref{s:intro}, the character $\chi_{\theta}$ is most properly thought
of as an element of the group ring of the weight lattice, but we
modify $\chi_{\theta}$ to be an element of $\C [\Lambda]$ by shifting
so that the term for the lowest weight is supported on the monomial
$\x^{0}\in \C [\Lambda]$; by abuse of notation we denote the resulting
polynomial in $\C [\Lambda]$ also by $\chi_{\theta}$.  With this
convention, the support of $\chi_{\theta}$ is contained in the
codominant cone $\Lambda^{+}$, and $\chi_{\theta}$ is actually a
polynomial under the identification $\C [\Lambda] \simeq \C
[x_{1}^{\pm 1},\dotsc ,x_{r}^{\pm 1}]$.  For example, if $\Phi =
A_{2}$ and $\theta = \varpi_{2}$, then $V_{\theta}$ is the standard
representation.  If we write $x = \x^{\alpha_{1}}, y=\x^{\alpha_{2}}$,
then $\chi_{\theta }=1+y+xy$.  Similarly if $\theta = \rho$, then
$V_{\theta}$ is the adjoint representation, and $\chi_{\theta} =
1+x+y+2xy+x^{2}y+y^{2}x+x^{2}y^{2}$.

\begin{definition}\label{def:toknum}
Let $V_{\theta}$ be an irreducible representation and let
$\chi_{\theta} (\x )$ be its character as above.  Then the
\emph{Tokuyama numerator} $N_{\theta} (\x) \in \C[q^{-1}] [\Lambda ]$
is the polynomial $N_{\theta} (\x) = \chi_{\theta} (\x) D(\x)$.
\end{definition}

Note that if $q=1$, then $D (\x) = \Delta (\x)$, and then by
\eqref{eq:wcf0} $N_{\theta} (\x)$ is a sum of signed monomials indexed
by the Weyl group $W$.  In general $N_{\theta} (\x)$ is a polynomial
supported on monomials $\x^{\beta}$ with $\beta$ a weight of
$V_{\theta +\rho}$.  When $\Phi = A_{r}$, Tokuyama showed how to write
$N_{\theta} (\x)$ as a sum over certain weights in the representation
$V_{\theta +\rho}$, and thus gave an explicit expression for the
numerator $N_{\theta} (\x)$ (cf.~Theorem \ref{thm:tok}).  The goal of this
paper is to give an explicit conjectural formula for the numerator
when $\Phi = G_{2}$.
 
\section{Crystal graphs and Littelmann Patterns} 

Recall that $\fg$ is a simple complex Lie algebra with root system
$\Phi$, $\theta$ is a dominant weight, and $V_{\theta}$ is the
irreducible representation of lowest weight $-\theta$.  Littelmann
patterns \cite{littelmann} provide a combinatorial way to index a
basis of $V_{\theta}$.  For instance when $\Phi = A_{r}$, Littelmann
patterns are essentially the famous Gel$'$fand--Cetlin patterns that
encode branching rules for $SL_{n}$ \cite{gt}.  In this section we
recall how to construct Littelmann patterns, with an emphasis on
$G_{2}$.

Littelmann patterns encode weight vectors of $V_{\theta}$ by
extracting data from the crystal graph $\BBB (\theta)$, so we begin by
discussing the latter.  We will not need much about crystal graphs and
refer to \cite{kashiwara} for a survey of their properties.  For our
purposes, we only need to know that $\BBB(\theta)$ is a finite
directed graph with edges colored by the simple roots $\Delta$.  The
vertices of $\BBB(\theta)$ are in bijection with certain weight
vectors in $V_\theta$; for $v\in\BBB(\theta)$ we write $v\mapsto \bar
v$.  Under this bijection, if there is an edge $v\rightarrow v'$
labelled by $\alpha \in \Delta$, then the weight of $\bar v$ is that
of $\bar v'$ plus $\alpha$.  Thus the edges correspond to the lowering
operators acting on $V_\theta$.  If we let $\theta \to \infty$, we
obtain an infinite graph $\BBB(\infty)$.  All the graphs
$\BBB(\theta)$ appear as subgraphs of $\BBB(\infty)$.

Now choose a reduced expression for the longest Weyl word $w_{0}$.
Littelmann proved that one can find a rational polyhedral cone
$C_{\infty}\subset \R^{N}$, where $N=|\Phi^{+}|$, such that the
lattice points $C_{\infty }\cap \Z^{N}$ are in bijection with the
vertices of $\BBB (\infty)$.  The inequalities defining the cone
$C_{\infty}$ depend only on $w_{0}$.  Furthermore, after choosing a
dominant weight $\theta$, one can find a second set of rational
inequalities depending on $\theta$ and $w_{0}$, such that if
$C_{\theta}\subset C_{\infty}$ denotes the corresponding cone, then
the lattice points $C_{\theta}\cap \Z^{N}$ are in bijection with the
vertices of $\BBB (\theta )$.  Finally he showed how to index these
lattice points using tables of nonnegative integers that record the
structure of certain paths in the crystal graph $\BBB (\theta)$.
These tables are the \emph{Littelmann patterns}; rather than giving
their definition in full generality, we explain how they work for
$G_{2}$ below and refer to \cite{littelmann} for more details.  Given
a Littelmann pattern $\pi$, we abuse notation and write $\pi \in \BBB
(\theta)$ to indicate that $\pi$ encodes a lattice point in
$C_{\theta}$ indexing a vertex of $\BBB (\theta)$.

We now specialize to $\Phi = G_{2}$.  The root system is shown in
Figure \ref{fig:g2} in \S\ref{s:gk}; we have $|\Phi^{+}| = 6$, and the
simple roots are $\alpha_{1},\alpha_{2}$.  The Weyl group has order
$12$ and the longest word $w_{0}$ has length $6$.  If we denote the
simple reflection corresponding to the simple root $\alpha_{i}$ by
$s_{i}$, then there are two reduced expressions for the longest word:
$s_{1}s_{2}s_{1}s_{2}s_{1}s_{2}$ and $s_{2}s_{1}s_{2}s_{1}s_{2}s_{1}$.
We will use the second expression.  A Littelmann pattern for $G_2$
then has the form
\begin{equation}\label{eq:pattern}
\left[\begin{array}{ccccc}
a&b&c&d&e\\
&&f&&\\
\end{array} \right]
\end{equation}
where $a,\dotsc ,f$ are integers, called the \emph{entries} of the
pattern.  To simplify notation, we usually write 
\begin{equation}\label{eq:patternshort}
[a,b,c,d,e][f]
\end{equation}
for
\eqref{eq:pattern}.

As described above, the entries are subject to certain inequalities
determined by our choice of reduced expression for $w_{0}$ and by the
highest weight $\theta$.  The first set of inequalities, which defines
the infinite cone $C_{\infty}$, gives lower bounds on the entries of a
pattern: we have
\begin{equation}\label{eq:circling}
2a\ge2b\ge c\ge2d\ge2e\ge0, \quad f\ge0. 
\end{equation}
We call these the \emph{circling}(\footnote{The terminology circling
and boxing comes from \cite{wmd5book}.}) inequalities; if any of these
is not strict, then we circle the entry in \eqref{eq:patternshort}
that appears on the left side of the corresponding inequality. Thus
$e$ and $f$ are circled if they vanish, $d$ is circled if it equals
$e$, and so on.  We indicate circling of an entry $u$ by a circle
superscript: $u^{\circ}$.

The second set of inequalities, which together with
\eqref{eq:circling} defines $C_{\theta}$, depends on the weight
$\theta $ and provides upper bounds on pattern entries.  Write $\theta
= \ell_ 1\varpi_1 + \ell_{2}\varpi_{2}$.  Then the entries must
satisfy
\begin{multline}\label{eq:boxing}
e \le \ell_1,
d \le \ell_2+e,
c \le \ell_1+3d-2e,
b \le \ell_2+c-2d+e,\\
a \le \ell_1+3b-2c+3d-2e,
f \le \ell_2+a-2b+c-2d+e.
\end{multline}
We say that an entry $u$ is boxed, denoted $\underline{u}$, if it
reaches its upper bound in \eqref{eq:boxing}. Thus we write
$\underline{e}$ if $e=\ell_1$, $\underline{d}$ if $d=\ell_2+e$, and so
on. To ease notation we sometimes give the boxing for a pattern in the
form of a pattern itself with entries restricted to $0$ and $1$ and
prefixed by $\bx$.  In such a pattern a 1 indicates that the
corresponding entry in the Littelmann pattern should be boxed, and $0$
indicates it should be unboxed.  For instance, the notation
$\bx[0,1,0,1,0][1]$ means a Littelmann pattern of the form
$[a,\underline{b},c,\underline{d},e][\underline{f}]$.

Each pattern $\pi$ determines a monomial $\x^{\pi}\in \C [\Lambda]
\simeq \C [x^{\pm 1},y^{\pm 1}]$: if $\pi=[a,b,c,d,e][f]$, then
$\x^{\pi} = x^{a+c+e}y^{b+d+f}$ (the variable $x$ corresponds to the
short simple root).  The pattern also determines a polynomial $H
(\pi)$ in $q^{-1}$:
\begin{definition}\label{def:stdcontrib}
Let $\pi$ be a boxed and circled Littelmann pattern.  Then the
\emph{standard contribution} $H (\pi)\in \Z [q^{-1}]$ of $\pi$ is
defined to be $H (\pi) = \prod_{u\in \pi} h (u)$, where the product is
taken over the entries $u$ of $\pi$, and
\[
h(u) = \left\{
\begin{array}{l l}
0 & \quad \text{if $u$ is both boxed and circled ($\underline{u}^\circ$),} \\
1& \quad \text{if $u$ is not boxed and is circled ($u^\circ$),} \\
-1/q & \quad \text{if $u$ is boxed and is not circled ($\underline{u}$),} \\
(1-1/q) & \quad \text{if $u$ is neither boxed nor circled ($u$).}
\end{array} \right. 
\]
\end{definition}

We call the function $H (\pi )$ the standard contribution of a boxed
and circled pattern $\pi$ because that is what a pattern contributes
in Tokuyama's original formula \cite{tokuyama}.  We state this formula
here for the convenience of the reader, and thus for the moment let
$\Phi$ be the root system $A_{r}$.  Fix a dominant weight $\theta =
\sum \ell_{i}\varpi_{i}$ and define $\chi_{\theta}$ as above.  The
reduced expression $w_{0}= s_{1} (s_{2}s_{1}) (s_{3}s_{2}s_{1}) \dotsb
$ determines a collection of circling and boxing inequalities; we
refer to \cite[Theorem 5.1, Corollary 1]{littelmann} for a complete
description (cf.~Example \ref{ex:a2}).  A pattern $\pi$ determines a
monomial $\x^{\pi}$, and we have the following theorem:
\begin{theorem}\label{thm:tok}
For $\Phi = A_{r}$ and with the standard contributions in Definition
\ref{def:stdcontrib}, we have
\begin{equation}\label{eq:tokstatement}
N_{\theta} (\x)= \chi_{\theta} (\x)D(\x) =\sum_{\pi \in \BBB (\theta +\rho)} H (\pi)\x^{\pi}.
\end{equation}
\end{theorem}

\begin{example}\label{ex:a2}
If $\Phi = A_{2}$, then patterns have the form $\pi = [a,b][c]$; for
such a $\pi$ we have $\x^{\pi}= x^{b+c}y^{a}$.  The circling
inequalities are $a\geq b\geq 0, c\geq 0$, and the boxing inequalities
are 
\[
b\leq \ell_{1}, a\leq \ell_{2}+b, c\leq \ell_{1}+a-2b.
\]
If we take $\theta = 0$, then $\chi_{0}=1$, thus
\eqref{eq:tokstatement} becomes a deformed version of the Weyl
denominator formula.  The sum is over the 8 patterns for
$\BBB (\rho)$:
\[
[0^{\circ },0^{\circ }][0^{\circ }], [0^{\circ },0^{\circ }][\ul{1}],
[\ul{1},0^{\circ }][0^{\circ }], [\ul{1},0^{\circ }][1],
[\ul{1},0^{\circ }][2], [1^{\circ },\ul{1}][\ul{0}^{\circ }],
[\ul{2},\ul{1}][0^{\circ }],[\ul{2},\ul{1}][\ul{1}].
\]
The standard contributions are 
\[
1, -1/q, -1/q, -(1/q) (1-1/q), (-1/q)^{2}, 0, (-1/q)^{2}, (-1/q)^{3},
\]
and one can check that $N_{\theta} (\x) = 1-q^{-1}x-q^{-1}y+
(q^{-2}-q^{-1}) xy +q^{-2} x^{2}y + q^{-2}xy^{2} -q^{-3}x^{2}y^{2} = D (\x)$.

\end{example}

\section{A Conjectural Tokuyama formula for $G_{2}$}

We now present our conjectural generalization of Tokuyama's theorem
for $G_{2}$.  As a first approximation, define the polynomial 
\begin{equation}\label{eq:guess}
\sum_{\pi \in \BBB (\theta +\rho )} H (\pi)\x^{\pi}.
\end{equation}
In other words, we simply take each pattern's contribution to be the
standard contribution from Definition \ref{def:stdcontrib}, where
boxing and circling are computed as in
\eqref{eq:circling}--\eqref{eq:boxing}.  One quickly sees that
\eqref{eq:guess} is not correct: \eqref{eq:guess} does not equal
$\chi_{\theta} (\x) D(\x) $.  On the other hand, \eqref{eq:guess} is
not that far from our goal: only certain coefficients in the sum are
wrong, and the corresponding monomials all contain at least one
pattern with a special form:

\begin{definition}\label{def:badmiddle}
A $G_{2}$-Littelmann pattern $[a,b,c,d,e][f]$ is called \emph{bad
middle} if $c=b+d$ and $b=d+1$.
\end{definition}
Note that whether or not a pattern is bad middle depends only its top
row, and is independent of the bottom row $[f]$.
We are now ready to give the main definition needed for our
conjecture.  

\begin{definition}\label{def:Hhat}
Let $\pi=[a,b,c,d,e][f]$ be a boxed and circled $G_{2}$ Littelmann pattern. We define
the contribution $\hat{H} (\pi)\in \Z [q^{-1}]$ as follows. 

First, if $\pi$ is not bad middle, or if $\pi$ is bad middle
but the boxing is not specified below, or if $\pi$ has an
entry that is both boxed and circled, then put $\hat{H}(\pi) =
H(\pi)$, the standard contribution of $\pi$.

Otherwise, we put $\hat{H}(\pi) =
\hat{T}(\pi')h(f)$, where $\pi '$ denotes the top row of $\pi$, and
$\hat{T}$ is defined as follows:
\begin{enumerate}
\item If $\pi '$ has boxing $\bx[0,0,1,0,0]$, then
we put $\hat{T}(\pi') = 0$. 
\item If $\pi'$ has boxing $\bx[1,0,1,0,0]$, then
we put
\[
\hat{T}(\pi') = \left\{
\begin{array}{l l}
0 & \quad \text{if $d=0$,} \\
T(\pi') & \quad \text{if $d>0$.}
\end{array} \right. 
\]
Here and in what follows we write $T (\pi')$ for the product of $h
(u)$ over the entries in the row $\pi ' \subset \pi$ (in other words,
this is what one would compute as the standard contribution of the top
row $\pi '$).
\item If $\pi'$ has boxing $\bx[1,0,0,0,0]$, then
we put
\[
\hat{T}(\pi') = \left\{
\begin{array}{l l}
(-q+1)/q^2 & \quad \text{if $e=0$ and $d=0$,} \\
(-q^3+2q^2-2q+1)/q^4 & \quad \text{if $e=0$ and $d>0$,} \\
T(\pi') & \quad \text{if $e>0$.}
\end{array} \right.
\]
\item If $\pi'$ has boxing $\bx[0,1,0,1,0]$, then
we put
\[
\hat{T}(\pi') = \left\{
\begin{array}{l l}
T(\pi') & \quad \text{if $a=b$,} \\
0 & \quad \text{if $b<a<c$ and $e=0$,} \\
(-q^2+2q-1)/q^5 & \quad \text{if $b<a<c-e$ and $e>0$,} \\
(q-1)/q^3 & \quad \text{if $a=c$ and $e=0$,} \\
(q^3-2q^2+2q-1)/q^5 & \quad \text{if $a=c-e$ and $e>0$,} \\
0 & \quad \text{if $a>c$ and $e=0$,} \\
(-q^2+2q-1)/q^5 & \quad \text{if $a>c-e$ and $e>0$.}
\end{array} \right. 
\]
\end{enumerate}
\item If $\pi'$ has boxing $\bx[0,0,0,0,0]$ and
$e=0$, then we put
\[
\hat{T}(\pi') = \left\{
\begin{array}{l l}
(q^2-2q+1)/q^2 & \quad \text{if $a=b$ and $d>0$,} \\
(q^3-3q^2+3q-1)/q^3 & \quad \text{if $b<a<c$ and $d>0$,} \\
(q^3-3q^2+4q-2)/q^3 & \quad \text{if $a=c$ and $d>0$,} \\
(q^3-3q^2+3q-1)/q^3 & \quad \text{if $a>c$ and $d>0$,} \\
(q-1)/q & \quad \text{if $a=b$ and $d=0$,} \\
(q^2-2q+1)/q^2 & \quad \text{if $a>b$ and $d=0$.}
\end{array} \right. 
\]

\item Finally, if $\pi'$ has boxing
$\bx[0,0,0,0,0]$ and $e>0$, then we put
\[
\hat{T}(\pi') = \left\{
\begin{array}{l l}
(q^4-3q^3+4q^2-3q+1)/q^4 & \quad \text{if $a=b$ and $d>e$,} \\
(q^5-4q^4+7q^3-7q^2+4q-1)/q^5 & \quad \text{if $a>b$ and $d>e$,} \\
(q^2-2q+1)/q^2 & \quad \text{if $a=b$ and $d=e$,} \\
(q^4-3q^3+4q^2-3q+1)/q^4 & \quad \text{if $a>b$ and $d=e$.}
\end{array} \right. 
\]
\end{definition}
We can now state our conjecture:

\begin{conj}\label{conj:g2conj}
Let $\Phi=G_2$ and put $N_{\theta} (\x) = \chi_{\theta} (\x) D (\x)$,
where $\chi_{\theta}$ is the character of the irreducible
representation $V_{\theta }$ of $G_{2}$ of lowest weight $-\theta$, shifted to be an element of $\C
[\Lambda]$ (as in the paragraph before Definition \ref{def:toknum}),
and where $D (\x) = \prod_{\alpha >0} (1-q^{-1}\x^{\alpha})$ is the
deformed Weyl denominator \eqref{eqn:def}.  Then we have
\begin{equation}\label{eq:tokg2statement}
N_\theta (\x )  = \sum_{\pi \in \BBB (\theta +\rho)}\hat{H}(\pi)\mb{x}^\pi.
\end{equation}
\end{conj}

Although we cannot currently prove Conjecture \ref{conj:g2conj}, we
have checked it in many cases by computer:
\begin{proposition}\label{prop:checking}
Conjecture \eqref{conj:g2conj} is true for all weights $\theta =
\ell_{1}\varpi_{1} + \ell_{2}\varpi_{2}$ with $0\leq \ell_{i} \leq
4$.  
\end{proposition}

\begin{example}\label{ex:g2rho}
Let $\theta =0$.  Then as in Example \ref{ex:a2} the identity
\eqref{eq:tokg2statement} becomes a deformed version of the Weyl
denominator identity.  The sum is taken over $64$ patterns.  On $24 $ of these
$\hat H$ vanishes since an entry is both boxed and circled.  Of the
remaining $40$, there are $12$ patterns that are bad middle, and $7$
of these have their contributions altered by Conjecture 
\ref{conj:g2conj}:
\begin{enumerate}
\item There are $2$ patterns with top row $[1^{\circ }, 1,
\underline{1}, 0^{\circ }, 0^{\circ }]$ and $3$ with top row
$[\underline{2}, 1, \underline{1}, 0^{\circ }, 0^{\circ }]$.  All of
these have $\hat H = 0$ (the first 2 by (1) and the second 3 by (2) in
Definition \ref{def:Hhat}).
\item There are $2$ patterns $[3, \underline{2}, 3, \underline{1},
0^{\circ }][0^{\circ }]$ and $[3, \underline{2}, 3, \underline{1},
0^{\circ }][\underline{1}]$.  Using (4) in Definition \ref{def:Hhat},
we compute that the first has $\hat H = (q-1)/q^{3}$ and the second
has $\hat H = -(q-1)/q^{4}$.
\end{enumerate}
\end{example}

\begin{remark}\label{rem:moredata}
We have checked \eqref{eq:tokg2statement} for larger weights than
those in Proposition \ref{prop:checking}.  The largest example we
checked was $\theta = 6\varpi_{1} + 6\varpi_{2}$.  For this example
the crystal graph $\BBB (\theta +\rho)$ has $262144$ vertices.
\end{remark}

\begin{remark}\label{rem:typeB}
The motivation to consider bad middle patterns comes from a similar
investigation by the third-named author into an analogue of Tokuyama's
theorem for the root system of type $B$ \cite{brconjecture}.  Indeed the
circling inequalities for the top row of the $G_{2}$-patterns are very
similar to those for type $B$ for a certain choice of reduced
expression for $w_0$.
\end{remark}

\section{Gindikin--Karpelevi\v c formula}\label{s:gk}

Let $F$ be a nonarchimedian local field with $\OOO$ its
valuation ring.  Let $\p$ be a uniformizer and let $q$ be the
cardinality of the residue field $\OOO/\p \OOO$. 

Let $G$ be a simply-connected split Chevalley group over $F$; for us
this will ultimately be of type $G_{2}$.  Let $T\subset B\subset G$ be
a maximal torus and a Borel subgroup.  Let $U^{-}$ be the opposite
unipotent radical to $B$.  Let $K\subset G$ be the maximal compact
subgroup $G (\OOO)$.

Let $\Phi$ be the root system of $G$ determined by $T$ and $B$, and
let $\Delta \subset \Phi$ be the corresponding simple roots.  As
before let $\Phi = \Phi^{+}\cup \Phi^{-}$ be the decomposition into
positive and negative roots.  For $\alpha \in \Phi $ let $e_{\alpha
}\colon F\rightarrow G$ be the generator of the root subgroup
corresponding to $\alpha$, and let $h_{\alpha} \colon F\rightarrow G$
be the coroot corresponding to $\alpha$. Thus $T$ is the subgroup
generated by $\{h_{\alpha}\mid \alpha \in \Delta \}$, $B$ is generated
by $T$ and $\{e_{\alpha}\mid \alpha >0 \}$, and $U^{-}$ is generated
by $\{e_{\alpha}\mid \alpha <0 \}$.

Now we introduce the ``spectral parameters.''  Let $\{z_{\alpha} \}$
be a set of nonzero complex numbers indexed by the simple roots.  Given any
root $\beta \in \Phi$, we define $\z^{\beta}\in \C$ by 
\[
\z^{\beta} = \prod_{\alpha \in \Delta} z_{\alpha}^{k_{\alpha}}, \quad
\text{where $\beta = \sum_{\alpha \in \Delta} k_{\alpha}\alpha$.}
\]
We can use the $\{z_{\alpha} \}$ to define a character $\chi \colon
T\rightarrow \C$ by putting
\[
\chi (\prod_{\alpha \in \Delta} h_{\alpha} (\p^{m_{\alpha}})) =
\prod_{\alpha \in \Delta} z_{\alpha}^{m_{\alpha }}, \quad m_{\alpha}\in \Z,
\]
and then declaring that $\chi$ is trivial on 
$T\cap K $.  We can extend $\chi$ to a character on $B$, and can then
define the \emph{principal series representation} $V_{\chi}$ by 
\[
V_{\chi} = \{f\colon G \rightarrow \C \mid f (bg) = \delta^{1/2}
(b)\chi (b)f (g), b\in B \}.
\] 
Here $\delta$ is the modular quasi-character of $B$, and the action of
$G$ is given by right translations: $(g\cdot f) (g') := f (g'g)$.  One
can prove that the space of $K$-invariant vectors $V_{\chi}^{K}$ is
one-dimensional.  We choose a nonzero element $\varphi_{K} \in
V_{\chi}^{K}$, called the \emph{spherical vector}, such that 
\[
\varphi_{K} (bk) = \delta^{1/2} (b)\chi (b), \quad b\in B, k\in K.
\]
We can now state the \emph{Gindikin--Karpelevi\v c formula}:

\begin{theorem}\label{thm:GKstatement}
We have 
\begin{equation}\label{eq:GKstatement}
\int_{U^{-} (F)} \varphi_{K} (u)\,du = \prod_{\alpha >0}\frac{1-q^{-1}\z^{\alpha}}{1-\z^{\alpha}}.
\end{equation}
\end{theorem}

We remark that Gindikin--Karpelevi\v c proved their formula for $F$
archimedian, in which case the right of \eqref{eq:GKstatement} becomes
a product of ratios of Gamma functions.  The formula for $F$
nonarchimedian was proved by Langlands \cite{langlands.eulerproducts}.

Now let $\C [[\Lambda^{+}]] \simeq \C [[x_{1},\dotsc ,x_{r}]]$ be the
formal power series ring on the codominant cone, and consider the generating function 
\begin{equation}\label{eq:doverdelta}
\frac{D (\x)}{\Delta (\x)} = \prod_{\alpha >0}
\frac{1-q^{-1}\x^{\alpha}}{1-\x^{\alpha }}\in \C [[\Lambda^{+}]].
\end{equation}
Up to a simple change of notation, \eqref{eq:doverdelta} coincides
with the right hand side of \eqref{eq:GKstatement}.  Our goal is to
express \eqref{eq:doverdelta} as a sum over the infinite crystal $\BBB
(\infty)$.  This was done in type $A$ by Bump--Nakasuji \cite{bn}, and
for all types by McNamara \cite{mcnamara} and independently by
Kim--Lee \cite{kim.lee}.  In type $A$ these three results are
equivalent, and take the form 
\begin{equation}\label{eq:bnstatement}
\frac{D (\x)}{\Delta (\x)} = \sum_{\pi \in \BBB (\infty)} H (\pi) \x^{\pi },
\end{equation}
where $H (\pi)$ is the standard contribution of a Littelmann pattern.
Our goal is now to prove the following theorem:

\begin{theorem}\label{thm:G2GK}
Let $\BBB (\infty)$ be the infinite crystal for $\Phi = G_{2}$.  Then 
we have 
\begin{equation}\label{eq:g2gkstatement}
\frac{D (\x)}{\Delta (\x)} = \sum_{\pi \in \BBB (\infty)} \hat{H} (\pi) \x^{\pi },
\end{equation}
where $\hat {H}$ is defined in Definition \ref{def:Hhat}.
\end{theorem}

Before we begin the proof, we need more notation.  Recall that a
\emph{vector partition} on the positive roots $\Phi^{+}$ is a function
$\xi \colon \Phi^{+}\rightarrow \Z_{\geq 0}$.  Define the \emph{index}
$\iota (\xi)$ of a vector partition to be the number of $\alpha \in
\Phi^{+}$ such that $\xi (\alpha ) \not = 0$.  Each vector partition
determines a monomial $\x^{\xi}\in \C [\Lambda^{+}]$ by $\x^{\xi} =
\x^{\beta}$, where
\begin{equation}\label{eq:sum}
\beta = \beta (\xi) := \sum_{\alpha >0}\xi
(\alpha)\alpha.
\end{equation}
We sometimes abuse notation and write a vector
partition as a sum as in \eqref{eq:sum}.

\begin{lemma}\label{lem:index}
We have 
\[
\frac{D (\x)}{\Delta (\x)} = \sum_{\xi} (1-q^{-1})^{\iota (\xi)}
\x^{\xi},
\]
where the sum is taken over all vector partitions on the positive roots.
\end{lemma}

\begin{proof}
This is a special case of \cite[Theorem 1.6]{kim.lee}. 
\end{proof}

\begin{lemma}\label{lem:bijection}
There is a bijection between the $G_{2}$-Littelmann patterns
satisfying the circling inequalities \eqref{eq:circling} and vector
partitions on the positive roots for $G_{2}$ such that if $\pi$ is
taken to the the partition $\xi$, then $\x^{\pi} = \x^{\xi}$.
\end{lemma}

\begin{proof}
Let $C = C_{\infty }\subset \R^{6}$ be the cone defined by \eqref{eq:circling}.  The
simplicial cone $C$ is generated by the points 
\begin{multline}\label{eq:pts}
v_{1} = (0,0,0,0,0,1), v_{2} = (1,0,0,0,0,0), v_{3} = (1,1,0,0,0,0),
\\
v_{3}' = (1,1,2,0,0,0), v_{5} = (1,1,2,1,0,0), v_{6} = (1,1,2,1,1,0),   
\end{multline}
and these are the primitive lattice points on the edges of $C$.  One
can check that $C$ is not unimodular; that is, the sublattice of
$\Z^{6}$ generated by the points \eqref{eq:pts} is not $\Z^{6}$, and
is in fact a sublattice of index $2$.  One can decompose $C$ as a
union of unimodular cones $C_1 \cup C_{2}$ by including the point
$v_{4} = (1,1,1,0,0,0)$.  In particular, we have
\[
C_{1} = \langle v_{1}, v_{2}, v_{3}, v_{4}, v_{5}, v_{6} \rangle
\quad \text{and}\quad C_{2} = \langle  v_{1}, v_{2}, v'_{3}, v_{4}, v_{5}, v_{6} \rangle.
\]
Thus any lattice point in $C$ can be uniquely written as a $\Z$-linear
combination of the points $v_{1},v_{2},v_{3},v_{3}',v_{4},
v_{5},v_{6}$, where on $C_{1}$ (resp.~$C_{2}$) we use all the $v_{i}$
except $v_{3}'$ (resp.~$v_{3}$).

Number the roots as in Figure \ref{fig:g2}.  We can use a lattice
point $v\in C$ to determine a vector partition as follows:
\begin{itemize}
\item If $v\in C_{1}$, then $a_{1}v_{1} + a_{2}v_{2} + a_{3}v_{3} +
a_{4}v_{4} + a_{5}v_{5} + a_{6}v_{6}$ determines $a_{1}\alpha_{1} +
a_{2}\alpha_{2} + a_{3}\alpha_{3} + a_{4}\alpha_{4} + a_{5}\alpha_{5}
+ a_{6} (\alpha_{3}+\alpha_{3}')$. 
\item If $v\in C_{2}$, then $a_{1}v_{1} + a_{2}v_{2} + a'_{3}v'_{3} +
a_{4}v_{4} + a_{5}v_{5} + a_{6}v_{6}$ determines $a_{1}\alpha_{1} +
a_{2}\alpha_{2} + a'_{3}\alpha'_{3} + a_{4}\alpha_{4} + a_{5}\alpha_{5}
+ a_{6} (\alpha_{3}+\alpha_{3}')$. 
\end{itemize}
Hence lattice points in $C_{1}\cap C_{2}$ correspond to partitions
$\xi $ such that $\xi (\alpha_{3}) = \xi (\alpha'_{3})$, whereas
points in $C_{1} \smallsetminus C_{1}\cap C_{2}$ (resp.~$C_{2}
\smallsetminus C_{1}\cap C_{2}$ ) correspond to partitions such that
$\xi (\alpha_{3})>\xi (\alpha_{3}')$, (resp.~$\xi (\alpha_{3})<\xi
(\alpha_{3}')$).  It is easy to check that this correspondence is a
bijection with the desired properties, which completes the proof of
the lemma.
\end{proof}
\begin{figure}[htb]
\begin{center}
\begin{tikzpicture}[scale=1]
\clip (-2.51,-2.51) rectangle (2.51,2.51);
\coordinate (0) at (0:0);
\coordinate (1) at (0:1);
\coordinate (2) at (60:1);
\coordinate (3) at (120:1);
\coordinate (4) at (180:1);
\coordinate (5) at (240:1);
\coordinate (6) at (300:1);
\coordinate (10) at (30:1.7320508075688772935274463415058723670);
\coordinate (20) at (90:1.7320508075688772935274463415058723670);
\coordinate (30) at (150:1.7320508075688772935274463415058723670);
\coordinate (40) at (210:1.7320508075688772935274463415058723670);
\coordinate (50) at (270:1.7320508075688772935274463415058723670);
\coordinate (60) at (330:1.7320508075688772935274463415058723670);
\draw[thick,->] (0) -- (1);
\draw[thick,->] (0) -- (2);
\draw[thick,->] (0) -- (3);
\draw[thick,->] (0) -- (4);
\draw[thick,->] (0) -- (5);
\draw[thick,->] (0) -- (6);
\draw[thick,->] (0) -- (10);
\draw[thick,->] (0) -- (20);
\draw[thick,->] (0) -- (30);
\draw[thick,->] (0) -- (40);
\draw[thick,->] (0) -- (50);
\draw[thick,->] (0) -- (60);

\node at (150:2.4) {$\alpha_{1}$};
\node at (0:1.4) {$\alpha_{2}$};
\node at (120:1.4) {$\alpha_{3}$};
\node at (60:1.4) {$\alpha_{4}$};
\node at (30:2.5) {$\alpha'_{3}$};
\node at (90:2.1) {$\alpha_{5}$};
\end{tikzpicture}
\end{center}
\caption{\label{fig:g2}}
\end{figure}

\begin{proof}[Proof of Theorem \ref{thm:G2GK}]
The proof is by a direct computation with the contributions $\hat {H}$
in Definition \ref{def:Hhat}.  Since the computation is easy but
lengthy, we give the main points of the argument and leave the details
to the reader.   

The vertices of $\BBB (\infty)$ in the sum in \eqref{eq:g2gkstatement}
are parameterized by unboxed Littelmann patterns; the only requirement
is that the entries of such a pattern satisfy the circling
inequalities \eqref{eq:circling}.  By Lemmas \ref{lem:index} and
\ref{lem:bijection}, if for any such pattern $\pi$ we had $\hat{H}
(\pi) = (1-q^{-1})^{\iota (\xi)}$, where $\xi$ is the vector
partition attached to $\pi$ in Lemma \ref{lem:bijection}, then Theorem
\ref{thm:G2GK} would follow immediately.

Unfortunately this is not the case: for many patterns $\pi$ the
contribution $\hat{H} (\pi)$ is quite different.  The simplest example
is the unboxed pattern $ \pi = [1,1,1,0,0][0]$.  According to
Lemma \ref{lem:bijection}, this corresponds to the vector partition
$1\cdot \alpha_{4}$.  On the other hand, we have $\hat {H} (\pi) =
(1-q^{-1})^{2}$.  Another example is provided by the pattern $\pi'
= [1,1,2,1,1][0]$.  We have $\hat {H} (\pi') = (1-q^{-1})$, yet the
vector partition is $1\cdot \alpha_{3} + 1\cdot \alpha'_{3}$.

However, in some sense these two patterns, which correspond to the
primitive generators of the rays $\R v_{4}$ and $\R v_{6}$, are the
main difficulty: all the patterns whose contributions under Definition
\ref{def:Hhat} and Lemma \ref{lem:bijection} disagree live in the
$4$-dimensional intersection $C' = C_{1}\cap C_{2} = \langle v_{1},
v_{2}, v_{4}, v_{5}, v_{6}\rangle $, and involve the rays generated by
$v_{4}$ and $v_{6}$ in an essential way.

More precisely, let us indicate the relative interiors of subcones of
the intersection $C'$ by subsets of $\{1,2,4,5,6 \}$.  Thus for
instance $\{2,4,6 \}$ means the subset of $C'$ of the form $\{av_{2} +
bv_{4} + cv_{6} \mid a,b,c \in \R_{>0}\}$; we abbreviate the notation
further by eliminating braces and commas and write simply $246$.  Then
investigation of Definition \ref{def:Hhat} shows that the only
subcones where (i) there is a discrepancy between $\hat{H} (\pi)$ and
$(1-q^{-1})^{\iota (\xi)}$, or (ii) $\hat {H} (\pi) \not = H (\pi)$
are those that appear in Table \ref{tab:cones}.  In this table a mark
in row ``vp'' (resp.~row ``corr'') indicates possibility (i)
(resp.~possibility (ii)).

To complete the proof of the theorem, one must systematically go
through Table \ref{tab:cones} and check that the corrections in
Definition \ref{def:Hhat} exactly compensate for the difference
between $\hat {H} (\pi)$ and $(1-q^{-1})^{\iota (\xi)}$.  
We illustrate this with the cones $4$ and $6$, which typify the
process. 

Consider the lattice points $av_{4}$ and $bv_{6}$, where $a,b\geq
1$.  The patterns (ignoring the bottom row, which plays no role) are
$\pi_{4}(a) := [a,a,a,0,0]$ and $\pi_{6} (b) = [b,b,2b,b,b]$.  Suppose
$a = 2b$ is even.  Then $\pi_{4} (a)$ and $\pi_{6} (b)$ contribute to
the same monomial, and their total contribution is $(1-q^{-1}) +
(1-q^{-1})^{2}$, which is what one expects from Lemma \ref{lem:index}.
Now suppose $a = 2b+1$ is odd.  If $a=1$, then there is an explicit
correction in Definition \ref{def:Hhat} that sets $\hat H (\pi_{4}
(1)) = (1-q^{-1})$.  If $a>1$, then the patterns $av_{4},
v_{4}+bv_{6}$ and $b (v_{2}+v_{5}) +v_{4} \in 245$ all contribute to
the same monomial.  According to Defintion \ref{def:Hhat} the pattern
$av_{4}$ contributes $(1-q^{-1})^{2}$, the pattern $v_{4}+bv_{6}$
contributes $(1-q^{-1})^{2}$ as well, and $b (v_{2}+v_{5}) +v_{4}$
contributes $1-3q^{-1}+4q^{-2}-2q^{-3}$.  Adding up all these
contributions one obtains $3-7q^{-1}+6q^{-2}-2q^{-3}$.  This exactly
equals the contributions of these patterns one wants from Lemma
\ref{lem:bijection}.  Indeed, when one computes the vector partitions
and their indices, one finds that these patterns \emph{should}
respectively contribute $1-q^{-1}$, $(1-q^{-1})^{3}$,
$(1-q^{-1})^{3}$.  Since $1-q^{-1} + (1-q^{-1})^{3} + (1-q^{-1})^{3} =
3-7q^{-1}+6q^{-2}-2q^{-3}$ we have perfect agreement.  Note this
computation has simultaneously accounted for (i) the ``vp'' and
``corr'' rows under $4$, (ii) the ``vp'' row under $6$, (iii) the ``corr'' row under $46$, and (iv)
the ``corr'' row under $245$ for those patterns in $245$ of the form
$rv_{2} + sv_{4} + tv_{5}$ with $r=t$.  

The remaining computations to complete Table \ref{tab:cones} are
entirely similar.  The most complicated case to check is the cone
$245$.  There the correction to the pattern corresponding to $av_{2} +
bv_{4} + cv_{5}$ depends on whether $a<c$, $a=c$, or $a>c$; as we saw
above we have already accounted for $a=c$.

\newcommand{\m}{$\bullet$}
\begin{table}[htb]
\begin{center}
\begin{tabular}{|c||c|c|c|c|c|c|c|c|c|c|c|c|}
\hline
&4&6&24&26&45&46&56&245&246&256&456&2456\\
\hline
vp&\m&\m &\m &\m &\m &&\m & \m & &\m &&\\
\hline
corr&\m &&\m &&\m &\m &&\m &\m &&\m&\m\\
\hline
\end{tabular}
\end{center}
\caption{Patterns where either $\hat {H}$
doesn't agree with the contribution computed from the bijection in
Lemma \ref{lem:bijection} (indicated by ``vp'') or where $H\not =
\hat {H}$ (indicated by ``corr'').\label{tab:cones}}
\end{table}
\end{proof}

\bibliographystyle{amsplain_initials_eprint}
\bibliography{g2conjecture}

\end{document}